\documentclass[11pt]{article}
\usepackage[utf8]{inputenc} 
\usepackage{amsmath,amssymb,amsthm,mathtools}
\usepackage[a4paper,margin=2.8cm]{geometry}
\usepackage{enumitem}
\usepackage{tikz}
\usepackage{hyperref}
\hypersetup{colorlinks=true,linkcolor=blue,citecolor=blue,urlcolor=blue}
\usepackage{authblk}

\newtheorem{theorem}{Theorem}[section]
\newtheorem{proposition}[theorem]{Proposition}
\newtheorem{lemma}[theorem]{Lemma}
\newtheorem{corollary}[theorem]{Corollary}

\theoremstyle{remark}
\newtheorem{remark}[theorem]{Remark}

\theoremstyle{definition}

\newcommand{\R}{\mathbb{R}}
\newcommand{\T}{\mathbb{T}}
\newcommand{\Z}{\mathbb{Z}}

\newcommand{\abs}[1]{\left\lvert #1\right\rvert}
\newcommand{\dd}{\,\mathrm{d}}

\title{\LARGE\bf A Variational Approach to Planar Choreographies via Ekeland’s Principle}

\author[1]{Juan Manuel Sánchez-Cerritos\thanks{Corresponding author: \texttt{jmsc@xanum.uam.mx}}}
\author[1]{Mayte Torres-Hernández}
\affil[1]{Departamento de Matemáticas, Universidad Autónoma Metropolitana-Iztapalapa, Ciudad de México, 09310, México}

\date{}

\begin{document}
\maketitle

\begin{abstract}
We present a variational approach to obtain periodic solutions of the $N$--body problem, in particular the “figure-eight’’ solution for three equal masses. 
The central idea is to explicitly optimize the \emph{spatial scale} within the Lagrangian action, leading to the functional 
$\mathcal F = K^{\alpha/(\alpha+2)} V^{2/(\alpha+2)}$.
We prove the existence of critical points of $\mathcal F$ that enforce a curve with a single self-crossing, and show that every reparametrized critical curve satisfies Newton’s equations and is free of collisions. 
This framework recovers the Chenciner–Montgomery “eight’’ (for $\alpha=1$) and extends to the whole range $0<\alpha<2$.
\end{abstract}

\section{Introduction and background}

The classical planar $N$--body problem, with positive masses $m_i>0$ and a homogeneous potential of degree $-\alpha$ ($0<\alpha\le 2$), describes the motion of point particles interacting through
\[
m_i\,\ddot q_i = \sum_{j\ne i} m_i m_j\,\frac{q_j - q_i}{|q_j-q_i|^{\alpha+2}},
\qquad i=1,\dots,N.
\]
The system is invariant under translations, rotations, and homotheties, which allows the dynamics to be reduced to the configuration space modulo these symmetries, and to focus on solutions that inherit specific geometric or temporal invariances.

\medskip
Among the most remarkable periodic solutions are the so-called \emph{choreographies}, in which all masses follow the same closed curve $\gamma(t)$, equally shifted in phase:
\[
q_i(t)=\gamma\!\left(t+\frac{2\pi i}{N}\right),\qquad i=0,\dots,N-1.
\]
For $N=3$ and equal masses, Chenciner and Montgomery~\cite{ChencinerMontgomery2000}
proved the existence of the celebrated \emph{figure-eight} orbit, a planar, collision-free loop along which the three bodies travel on the same path with a phase shift.
Their proof is variational and relies on Palais’ principle of symmetric criticality to restrict the Lagrangian action
\(
\mathcal A(x)=K(x)+V(x)
\)
to a suitable class of symmetric curves.

\medskip
The variational study of periodic solutions in the $N$--body problem goes back to the Maupertuis–Jacobi principle of least action and to the functional-analytic methods developed for Hamiltonian systems. 
A modern account of these ideas can be found in
\cite{ArnoldMMCM,AmbrosettiCotiZelati,Ekeland1990},
where the existence and compactness conditions for periodic orbits in Sobolev spaces are established.
Likewise, Gordon’s \emph{strong-force criterion}~\cite{Gordon1975}
provides an early characterization of singular potentials that still allow collision avoidance through variational minimization.

\medskip
Subsequent authors have proposed simplified or alternative formulations of the Chenciner– Montgomery scheme. 
In particular, Zhang and Zhou~\cite{ZhangZhou2002} gave a rigorous and elementary proof of the existence of a figure-eight choreography for three equal masses, based on minimizing the action within a class of odd and reflection-symmetric functions, combined with energy estimates and a local deformation argument in the spirit of Marchal~\cite{Marchal2002}.
These works firmly established the variational approach as a principal tool for constructing periodic solutions of the $N$--body problem (see also \cite{FerrarioTerracini2004,BarutelloTerracini2011,ChenOuyang2001}).
For a classification of symmetry groups in planar choreographies, see \cite{MontaldiSteckles2013}, 
and for broader expositions on the role of symmetry in periodic orbits,
\cite{ChencinerICM2002,ChencinerICMP2003}.

\medskip
A key feature of these methods is the \emph{homogeneity} of the potential, which allows the spatial scale to be separated as an independent degree of freedom.  
Indeed, for a potential of degree $-\alpha$, the Lagrangian action
\(
\mathcal A(x)=K(x)+V(x)
\)
satisfies the scaling laws $K(x_\lambda)=\lambda^2K(x)$ and $V(x_\lambda)=\lambda^{-\alpha}V(x)$.
This naturally suggests optimizing in $\lambda>0$ and introducing an intrinsic functional of the form
\[
\mathcal F(x)=K(x)^{\frac{\alpha}{\alpha+2}}\,V(x)^{\frac{2}{\alpha+2}},
\]
which is invariant under spatial homotheties.
The functional $\mathcal F$ coincides, up to a constant factor, with the minimum of the homothetic action 
$\Phi_\alpha(\lambda;x)=\lambda^2K(x)+\lambda^{-\alpha}V(x)$,
and thus gives a direct characterization of the action minimizers once the optimal scale has been eliminated.
This viewpoint unifies the traditional steps of the variational method and allows one to work conveniently on natural submanifolds such as $\{K=1\}$ or $\{V=1\}$ within the Sobolev space $H^1(\T)$
(cf.~\cite{AmbrosettiCotiZelati}).  In particular, this eliminates the need to fix $K=1$ or $V=1$ a priori, allowing the direct use of $\mathcal F$ and Ekeland’s principle without intermediate rescalings.

\medskip
In the literature, the elimination of the scale is usually implemented by fixing either $K=1$ or $V=1$ and minimizing the complementary term in a symmetric class.  
Here we take as the primary object the scale-invariant functional
\[
\mathcal F=K^{\frac{\alpha}{\alpha+2}}V^{\frac{2}{\alpha+2}},
\]
which (up to a constant) equals the minimum of the homothetic action
$\min_{\lambda>0}\Phi_\alpha(\lambda;\cdot)$.
Working directly with $\mathcal F$ makes the analysis much clearer. 
It allows one to apply Ekeland’s variational principle naturally on $\mathcal M=\{K=1\}$ and to extend Marchal’s argument against collisions within this invariant framework. 
This approach also brings out the link with Newton’s equations through a global virial identity, while the Fourier symmetries $(A_x)$ and $(Y^\star)$ yield a stronger Poincaré inequality that guarantees coercivity and compactness in $\mathcal H_{\mathrm{eight}}$.

\medskip
This unified approach yields, for $N=3$ and $0<\alpha<2$, a $2\pi$–periodic, collision-free solution possessing the symmetries of a vertical figure-eight. 
The case $\alpha=1$ reproduces the Newtonian result of Chenciner and Montgomery, while the full range $0<\alpha<2$ gives a continuous family of choreographies with the same geometric structure.

\medskip
The paper is organized as follows. 
Section~\ref{secc:planteamiento} sets up the variational framework. 
In Section~\ref{sec:apriori} we introduce the scale-invariant functional and show its equivalence with the envelope of the homothetic action $\Phi_\alpha(\lambda;x)$.
Section~\ref{sec:ekeland} applies Ekeland’s variational principle within the symmetric class
$\mathcal H_{\mathrm{eight}}$ defined by the symmetries $(A_x)$ and $(Y^\star)$.
Section~\ref{subsec:ausencia-colisiones} establishes collision avoidance following Marchal’s argument (see also \cite{Gordon1975}),
and Section~\ref{sec:fig8} constructs the smooth, collision-free periodic choreography of figure-eight type.

\section{Variational setting}
\label{secc:planteamiento}

We fix the period to be $T=2\pi$ and denote $\T:=\R /(2\pi\Z)$. 
In what follows we consider three equal masses $m>0$.
For $x=(x_1,\dots,x_N):\T\to (\R^2)^N$ we define
\[
K(x):=\frac12\int_0^{2\pi}\sum_{i=1}^N m\,\abs{\dot x_i(t)}^2\,\dd t,
\qquad 
V(x):=\int_0^{2\pi} U\big(x(t)\big)\,\dd t,
\qquad 
\mathcal A(x):=K(x)+V(x),
\]
where $U(x)=\sum_{1\le i<j\le N} m^2\,\abs{x_i-x_j}^{-\alpha}$ is homogeneous of degree $-\alpha$.

\medskip
A \emph{choreography} is a loop $\gamma:\T\to\R^2$ such that
\begin{equation}\label{eq:choreo}
 x_i(t)=\gamma\!\Big(t+\frac{2\pi i}{N}\Big),\qquad i=0,\dots,N-1.
\end{equation}
In the case $N=3$, we take $i=0,1,2$ and phase shifts of $2\pi/3$.

\subsection{Symmetric class $(A_x)$–$(Y^\star)$ for the “eight’’ (case $N=3$)}

We fix the self–intersection at the origin,
\[
\gamma(0)=(0,0),
\]
and impose the following two symmetries:
\begin{itemize}[leftmargin=1.6em]
  \item[(A$_x$)] Semi–antiperiodicity:
  \[
  \gamma(t+\pi)=(x(t),-y(t)) \quad\Rightarrow\quad x(t+\pi)=x(t),\;\; y(t+\pi)=-y(t).
  \]
  \item[(Y$^\star$)] Reflection symmetry about the $y$–axis centered at $t=\frac{\pi}{2}$:
  \[
  \gamma\!\Big(\frac{\pi}{2}+t\Big)=\big(-x(\tfrac{\pi}{2}-t),\,y(\tfrac{\pi}{2}-t)\big),
  \]
  that is, $x(\frac{\pi}{2}+t)=-x(\frac{\pi}{2}-t)$ (odd with respect to $t=\frac{\pi}{2}$)
  and $y(\frac{\pi}{2}+t)=y(\frac{\pi}{2}-t)$ (even with respect to $t=\frac{\pi}{2}$).
\end{itemize}

\medskip
Let \(f:\T\to\R\) be a \(2\pi\)-periodic function (for instance, a component \(x\) or \(y\) of \(\gamma\)).
Its Fourier expansion is
\[
f(t)=a_0+\sum_{k\ge 1}\big(a_k\cos(kt)+b_k\sin(kt)\big),
\]
with real coefficients \(a_k,b_k\).
For the vector function \(\gamma=(x,y)\) this decomposition is applied componentwise.

\begin{lemma}[Zero mean of $x$ from $Y^\star$]
\label{lem:media-x-cero}
One has $\displaystyle \int_0^{2\pi} x(t)\,dt=0$. In particular, the constant Fourier coefficient of $x$ vanishes.
\end{lemma}

\begin{proof}
Define $\tilde x(u):=x\!\big(\frac{\pi}{2}+u\big)$ for $u\in[-\frac{\pi}{2},\frac{\pi}{2}]$.
By $(Y^\star)$, $\tilde x(-u)=-\tilde x(u)$, i.e.\ $\tilde x$ is odd on a symmetric interval around $0$. 
Hence
\[
\int_0^{\pi} x(t)\,dt=\int_{-\pi/2}^{\pi/2} \tilde x(u)\,du=0.
\]
Using now $(A_x)$, we have $x(t+\pi)=x(t)$, so that
$\int_\pi^{2\pi} x(t)\,dt=\int_0^\pi x(t)\,dt=0$,
and adding both halves gives the result.
\end{proof}

\begin{remark}[Pointwise consequences at $t=\frac{\pi}{2}$]
From $(Y^\star)$ it follows that $x(\frac{\pi}{2})=-x(\frac{\pi}{2})$, thus $x(\frac{\pi}{2})=0$.
Differentiating, one gets $y'(\frac{\pi}{2})=0$, which means the curve crosses the $y$–axis orthogonally at the top of the upper lobe.
\end{remark}

\paragraph{Effect of the symmetries.}

\begin{proposition}[Allowed Fourier spectrum]
\label{prop:fourier-ax-y}
Under the symmetries $(A_x)$ and $(Y^\star)$ defined above, 
the $x$–component has only even modes and no constant term,
while the $y$–component has only odd modes:
\[
x(t)=\sum_{k\ge 1} a_{2k}\cos(2k\,t),\qquad
y(t)=\sum_{k\ge 0} b_{2k+1}\sin\big((2k+1)t\big).
\]
\end{proposition}

\begin{proof}
From $(A_x)$, $x(t+\pi)=x(t)$, hence the effective period of $x$ is $\pi$, and its Fourier series on $[0,2\pi]$ contains only even multiples of the fundamental frequency. 
By Lemma~\ref{lem:media-x-cero}, the coefficient $a_0$ vanishes.
Similarly, $(A_x)$ gives $y(t+\pi)=-y(t)$, so $y$ is $\pi$–antiperiodic and its series involves only odd harmonics.
\end{proof}

\begin{remark}[Option (NC1) on $y$]
If one further imposes the condition “no first harmonic’’ on $y$ (that is, $b_1=0$), 
then the first allowed mode for $y$ becomes $k=3$.
\end{remark}

\section{A scale–invariant functional and its scale envelope}
\label{sec:apriori}

We start from the Lagrangian action
\[
\mathcal A(x)=K(x)+V(x),
\]
and consider the spatial homothetic family $x_\lambda(t):=\lambda\,x(t)$, $\lambda>0$ (time is kept fixed).
We define the \emph{scale envelope}
\begin{equation}\label{eq:defPhiA}
\Phi_\alpha(\lambda;x):=\mathcal A(x_\lambda)
=\lambda^2 K(x)+\lambda^{-\alpha}V(x).
\end{equation}

For $x\in H^1\!\big([0,2\pi];(\R^2)^N\big)$ and $x_\lambda(t)=\lambda x(t)$ one has
\[
K(x_\lambda)=\lambda^2 K(x),\qquad 
V(x_\lambda)=\lambda^{-\alpha} V(x).
\]

\medskip

If $x$ has collisions, then $V(x)$ may be $+\infty$. We therefore work in the open set $\Omega$ of collision–free curves, where $V$ is finite and $C^1$.

\medskip

Motivated by the scaling laws, we introduce
\[
\mathcal F(x):=K(x)^{\frac{\alpha}{\alpha+2}}\,V(x)^{\frac{2}{\alpha+2}},
\qquad
p:=\frac{\alpha}{\alpha+2},\ \ q:=\frac{2}{\alpha+2}.
\]

\medskip

For $x\in\Omega$, differentiating \eqref{eq:defPhiA} gives
\[
\frac{d}{d\lambda}\Phi_\alpha(\lambda;x)
=2\lambda K-\alpha \lambda^{-\alpha-1}V.
\]
The optimizer satisfies
\(
\lambda^{\alpha+2}=\dfrac{\alpha V}{2K}.
\)
Substituting into \(\Phi_\alpha\),
\[
\min_{\lambda>0}\Phi_\alpha(\lambda;x)
=\Big(1+\frac{2}{\alpha}\Big)\lambda^{2}K
= C_\alpha\,K(x)^{\frac{\alpha}{\alpha+2}}\,V(x)^{\frac{2}{\alpha+2}}
= C_\alpha\,\mathcal F(x),
\]
with, for instance,
\[
C_\alpha=\frac{\alpha+2}{2}\left(\frac{\alpha}{2}\right)^{-\frac{\alpha}{\alpha+2}}.
\]
In particular,
\begin{equation}\label{eq:minPhiF}
\min_{\lambda>0}\Phi_\alpha(\lambda;x)=C_\alpha\,\mathcal F(x).
\end{equation}

\subsubsection*{Minimization equivalences}
Let $\mathcal H$ be a class of curves stable under homotheties (e.g.\ $H^1(\T)$ or the symmetric class $\mathcal H_{\mathrm{eight}}$).

\begin{proposition}[Equivalences]\label{prop:equivalencias}
\begin{enumerate}[label=(\roman*),leftmargin=1.4em]
\item $\displaystyle \inf_{x\in\mathcal H}\mathcal F(x)=\inf\{V(x)^q:\ x\in\mathcal H,\ K(x)=1\}$.
\item $\displaystyle \inf_{x\in\mathcal H}\mathcal F(x)=\inf\{K(x)^p:\ x\in\mathcal H,\ V(x)=1\}$.
\end{enumerate}
\end{proposition}

\begin{proof}
If $x$ is not constant, taking $\lambda=K(x)^{-1/2}$ yields $K(x_\lambda)=1$ and 
$\mathcal F(x)=V(x_\lambda)^q$. Similarly, choosing $\lambda=V(x)^{1/\alpha}$ gives $V(x_\lambda)=1$ and $\mathcal F(x)=K(x_\lambda)^p$.
\end{proof}

\begin{lemma}[Scale invariance]
For every $\lambda>0$, $\mathcal F(x_\lambda)=\mathcal F(x)$.
\end{lemma}
\begin{proof}
Since $K(x_\lambda)=\lambda^2K(x)$ and $V(x_\lambda)=\lambda^{-\alpha}V(x)$, we have
$\mathcal F(x_\lambda)=\lambda^{2p-\alpha q}\mathcal F(x)=\mathcal F(x)$, because $2p-\alpha q=0$.
\end{proof}

\subsection{From the scale–invariant functional to Newton’s equation}

We work on the collision–free open set
\[
  \Omega:=\{x\in\mathcal H:\ x(t)\ \text{has no collisions for all }t\in[0,2\pi]\}.
\]
On $\Omega$ the potential $U$ is $C^\infty$, hence $K,V\in C^1(\Omega)$.

Let $\mathcal H_{\mathrm{eight}}$ be the symmetric class determined by $(A_x)$ and $(Y^\star)$.
We denote its tangent space by
\[
  T_x\mathcal H_{\mathrm{eight}}
  :=\{\eta\in H^1:\ \eta \text{ satisfies, to first order, the same symmetries and linear constraints as }x\}.
\]

\begin{remark}[Nonemptiness of the symmetric class]
\label{rem:no-vaciedad}
The class $\mathcal{H}_{\mathrm{eight}}$ is nonempty. 
For instance,
\[
\gamma(t)=(\sin 2t,\sin t),\qquad t\in[0,2\pi],
\]
satisfies
\[
\gamma(t+\pi)=(\sin 2t,-\sin t)\quad\text{(that is, $(A_x)$),}
\]
and
\[
\gamma\Big(\tfrac{\pi}{2}+t\Big)=\big(\sin(2t+\pi),\sin(t+\tfrac{\pi}{2})\big)
   =\big(-\sin 2t,\cos t\big)
   =\big(-x(\tfrac{\pi}{2}-t),\,y(\tfrac{\pi}{2}-t)\big),
\]
since $x(\tfrac{\pi}{2}-t)=\sin(\pi-2t)=\sin 2t$ and 
$y(\tfrac{\pi}{2}-t)=\sin(\tfrac{\pi}{2}-t)=\cos t$. 
The trace has two symmetric lobes crossing transversally at the origin, forming a vertical “8”. 
For three equal masses, $x_i(t)=\gamma(t+\tfrac{2\pi i}{3})$ is a choreography in $\mathcal H_{\mathrm{eight}}$; 
moreover, for a small rescaling $\varepsilon\,\gamma$ there are no collisions (phase shift $2\pi/3$).
\end{remark}

With this preliminary check (nonemptiness) in hand, we compute the first variation of $\mathcal F$
within $\mathcal H_{\mathrm{eight}}$, using admissible variations 
$\eta\in T_x\mathcal H_{\mathrm{eight}}$ (a closed linear subspace).

For admissible $\eta=(\eta_i)$ we obtain
\begin{align}
  \delta K(x)[\eta]
   &= \frac12\int_0^{2\pi}\sum_i m\,\frac{d}{d\varepsilon}\Big|_{\varepsilon=0}\!|\dot x_i+\varepsilon\dot\eta_i|^2\,dt
    = \int_0^{2\pi}\sum_i m\,\dot x_i\cdot\dot\eta_i\,dt \notag\\
   &= -\int_0^{2\pi}\sum_i m\,\ddot x_i\cdot \eta_i\,dt,
   \label{eq:deltaK-strong}
\end{align}
by integration by parts and periodicity (boundary terms vanish).
For the potential,
\begin{equation}\label{eq:deltaV-strong}
  \delta V(x)[\eta]
  = \int_0^{2\pi}\sum_i \nabla_{x_i}U(x(t))\cdot \eta_i(t)\,dt,
\end{equation}
since $U$ is $C^1$ away from collisions. By the chain rule,
\begin{equation}\label{eq:D-F-strong}
  D\mathcal F(x)[\eta]
  = p\,K^{p-1}V^{q}\,\delta K(x)[\eta]
    + q\,K^{p}V^{q-1}\,\delta V(x)[\eta].
\end{equation}

If $x$ is critical for $\mathcal F$ on $\Omega\cap\mathcal H_{\mathrm{eight}}$, 
then $D\mathcal F(x)[\eta]=0$ for every $\eta\in T_x\mathcal H_{\mathrm{eight}}$. 
Using \eqref{eq:deltaK-strong}–\eqref{eq:D-F-strong} we obtain
\[
  \int_0^{2\pi}\sum_i\Big(-pK^{p-1}V^{q}\,m\,\ddot x_i
  + qK^{p}V^{q-1}\,\nabla_{x_i}U(x)\Big)\cdot \eta_i\,dt=0.
\]
Comparing coefficients of $\eta$ in \eqref{eq:D-F-strong}, one infers the existence of scalar multipliers $c_1$ and $c_2$ such that

\begin{equation}\label{eq:EL-weak-final}
  -\,c_1\,\ddot x_i(t)+c_2\,\nabla_{x_i}U\big(x(t)\big)=0,
  \qquad c_1:=pK^{p-1}V^q m,\; c_2:=qK^{p}V^{q-1}.
\end{equation}

\begin{remark}[Constraint qualification]
\label{rem:cualificacion}
On $\mathcal{M}=\{K=1\}$, the differential of $K$ does not vanish. 
Indeed, for $x\in\mathcal{M}$ and any $\eta\in H^1$,
\[
DK(x)[\eta] = \int_0^{2\pi} \sum_i m\,\dot x_i\cdot\dot\eta_i \,\dd t.
\]
If $DK(x)=0$, taking $\eta=x$ gives $2K(x)=0$, contradicting $K=1$. 
Therefore $DK(x)\neq 0$ on $\mathcal{M}$, so the qualification hypothesis of the Lagrange multiplier theorem holds.
\end{remark}

Setting $\rho:=c_2/c_1>0$, the previous equation becomes
\begin{equation}\label{eq:newton-lambda-final}
  m\,\ddot x_i(t)=\rho\,\nabla_{x_i}U\big(x(t)\big),
\end{equation}
a reparametrized form of Newton’s equations.

Multiplying \eqref{eq:newton-lambda-final} by $x_i$, summing in $i$, and integrating over $[0,2\pi]$ we get
\[
  \sum_i m\int_0^{2\pi}\ddot x_i\cdot x_i\,dt
  = -2K(x), \qquad
  \sum_i\int_0^{2\pi}\nabla_{x_i}U(x)\cdot x_i\,dt=-\alpha\,V(x),
\]
(the first identity uses periodicity; the second uses the homogeneity of $U$).
Comparing both identities yields
\begin{equation}\label{eq:virial-final}
  2K(x)=\alpha\,\rho\,V(x),
\end{equation}
which shows that $\rho=\frac{2K}{\alpha V}$ is constant (it depends only on the orbit integrals).
With the change of variables $s=\sqrt{\rho}\,t$ we reach the standard form
\begin{equation}\label{eq:newton-standard-final}
  m\,\frac{d^2x_i}{ds^2}=\nabla_{x_i}U\big(x(s)\big),
\end{equation}
namely, Newton’s classical equation.

\begin{remark}[Equivalence of critical points of $\mathcal A$ and $\mathcal F$]
Although the functionals $\mathcal A$ and $\mathcal F$ are not identical, 
they have the same critical points up to time–reparametrization.
Indeed, every critical point of $\mathcal A$ is mapped to a critical point 
of $\mathcal F$ by choosing the minimising scale $\lambda>0$, and conversely
every critical point of $\mathcal F$ yields, after the rescaling 
$s=\lambda t$, a solution of Newton’s equation. 
Thus $\mathcal F$ is a scale–invariant reformulation of the classical 
action functional $\mathcal A$, adapted to the variational treatment of choreographies.
\end{remark}

\begin{remark}[Regularity and energy]
Away from collisions, $U$ is smooth; therefore any weak solution of \eqref{eq:newton-standard-final}
is smooth by bootstrap. Moreover, the mechanical energy
\(E=\tfrac12\sum_i m|\dot x_i|^2-U(x)\)
is conserved in the reparametrized time $s$.
\end{remark}

\begin{proposition}[Reinforced Poincaré inequality for the vertical “8’’]
\label{prop:poincare-vertical}
Let $\gamma=(x,y)\in H^1(\T;\R^2)$ be a $2\pi$–periodic curve satisfying
\[
(A_x)\ \ \gamma(t+\pi)=(x(t),-y(t)),\qquad
(Y^\star)\ \ \gamma\!\Big(\tfrac{\pi}{2}+t\Big)=\big(-x(\tfrac{\pi}{2}-t),\,y(\tfrac{\pi}{2}-t)\big).
\]
Then
\[
\int_{0}^{2\pi}\!|\gamma(t)|^2\,dt\ \le\ \int_{0}^{2\pi}\!|\dot\gamma(t)|^2\,dt.
\]
\end{proposition}

\begin{proof}
\emph{(Fourier structure).}
From $(A_x)$ we have $x(t+\pi)=x(t)$ and $y(t+\pi)=-y(t)$, hence $x$ contains only \emph{even} harmonics and $y$ only \emph{odd} ones.
Moreover, $(Y^\star)$ says that $x$ is odd and $y$ even with respect to $t=\tfrac{\pi}{2}$; shifting the Fourier series to that center removes cosines in $x$ and sines in $y$ (in the shifted sense). Returning to the origin, we may write with real coefficients
\[
x(t)=\sum_{\substack{k\ge 2\\ k \text{ even}}} \tilde a_k\,\sin(kt),\qquad
y(t)=\sum_{\substack{k\ge 1\\ k \text{ odd}}} \tilde b_k\,\sin(kt).
\]
In particular, $x$ has no constant term (no $k=0$ mode).

\medskip
\emph{(Parseval and componentwise bounds).}
For $2\pi$–periodic functions,
\[
\int_{0}^{2\pi}\!|f|^2\,dt=\pi\sum_{k\in\Z}\big(|a_k|^2+|b_k|^2\big),\qquad
\int_{0}^{2\pi}\!|f'|^2\,dt=\pi\sum_{k\in\Z}k^2\big(|a_k|^2+|b_k|^2\big).
\]
Applying this to $x$ (only even sines with $k\ge 2$) gives
\[
\int_{0}^{2\pi}\!|x|^2
=\pi\sum_{\substack{k\ge 2\\ k \text{ even}}} |\tilde a_k|^2
\le \frac{\pi}{4}\sum_{\substack{k\ge 2\\ k \text{ even}}} k^2 |\tilde a_k|^2
=\frac{1}{4}\int_{0}^{2\pi}\!|x'|^2,
\]
since $k^2\ge 4$ for even $k\ge 2$. For $y$ (only odd sines with $k\ge 1$),
\[
\int_{0}^{2\pi}\!|y|^2
=\pi\sum_{\substack{k\ge 1\\ k \text{ odd}}} |\tilde b_k|^2
\le \pi\sum_{\substack{k\ge 1\\ k \text{ odd}}} k^2 |\tilde b_k|^2
=\int_{0}^{2\pi}\!|y'|^2,
\]
since $k^2\ge 1$ for odd $k\ge 1$.

\medskip
\emph{(Vector sum).}
Adding and using $|\gamma|^2=|x|^2+|y|^2$, $|\dot\gamma|^2=|x'|^2+|y'|^2$,
\[
\int_{0}^{2\pi}\!|\gamma|^2
\le \frac{1}{4}\int_{0}^{2\pi}\!|x'|^2+\int_{0}^{2\pi}\!|y'|^2
\le \int_{0}^{2\pi}\!\big(|x'|^2+|y'|^2\big)
=\int_{0}^{2\pi}\!|\dot\gamma|^2.
\]
\end{proof}

\medskip
\noindent
As an immediate consequence, for equal masses $m>0$ one obtains an explicit kinetic coercivity:

\begin{corollary}[Coercivity in terms of $K$]
\label{cor:coercividad-K}
If $\gamma=(x,y)\in\mathcal H_{\mathrm{eight}}$ satisfies $(A_x)$ and $(Y^\star)$, then
\[
\int_{0}^{2\pi}\!|\gamma(t)|^2\,dt
\ \le\
\int_{0}^{2\pi}\!|\dot\gamma(t)|^2\,dt
\ =\ \frac{2}{m}\,K(\gamma).
\]
In particular, on the sublevel $K=1$ one has the uniform bound
\[
\|\gamma\|_{L^2}^2\ \le\ \frac{2}{m}.
\]
\end{corollary}

\begin{proof}
The identity $\int |\dot\gamma|^2 = \tfrac{2}{m}K(\gamma)$ follows from the definition of the kinetic energy,
\[
K(\gamma)=\frac{m}{2}\int_0^{2\pi}|\dot\gamma(t)|^2\,dt,
\]
and the inequality is precisely Proposition~\ref{prop:poincare-vertical}.
\end{proof}

If one also imposes \emph{(NC1)} on $y$ (removing the $k=1$ mode), then
\(\int|y|^2\le \tfrac{1}{9}\int|y'|^2\).
The global constant remains $1/4$ due to the $x$–component, whose first allowed mode is $k=2$.

\medskip

By Proposition~\ref{prop:equivalencias}, minimizing $\mathcal F$ is equivalent to minimizing $V^q$ on $\mathcal M=\{K=1\}$ within $\mathcal H_{\mathrm{eight}}$.
 
The tangent gradient of $V^q$ along $\mathcal M$ is the orthogonal projection (for the $H^1$ metric) 
of the linear functional \eqref{eq:deltaV-strong} onto $T_x\mathcal M$, and the Lagrange multiplier enforcing $K=1$ turns out to be precisely $\rho=2K/(\alpha V)$.

\medskip

The previous result shows that, within the symmetric class $\mathcal H_{\mathrm{eight}}$,
the kinetic functional $K$ controls the $L^2$ norm of the trajectory, i.e.\ the kinetic energy enforces geometric coercivity on the curve.
In particular, on the sublevel $K=1$ the family of curves is uniformly bounded in $H^1(\T)$.

Together with the closed character of the symmetries $(A_x),(Y^\star)$
and the compact embedding $H^1(\T)\hookrightarrow C^0(\T)$,
this yields the completeness and lower semicontinuity properties required to apply Ekeland’s principle
to the functional
\[
f:=\mathcal F|_{\mathcal M}=V^q,
\qquad
\mathcal M=\{x\in\mathcal H_{\mathrm{eight}}:\ K(x)=1\}.
\]
We proceed to state these properties.

\medskip

The coercivity established in Proposition~\ref{prop:poincare-vertical}
guarantees that $\mathcal M$ is bounded in $H^1$, and in what follows
we verify the analytical conditions needed to apply Ekeland’s principle to
\[
f(x)=V(x)^q,\qquad q=\tfrac{2}{\alpha+2}.
\]

\begin{lemma}[Completeness and boundedness of $\mathcal M$]
\label{lem:completez-M}
With $\mathcal H_{\mathrm{eight}}$ the symmetric class and $\mathcal M=\{x\in\mathcal H_{\mathrm{eight}}:K(x)=1\}$,
one has: (i) $\mathcal M$ is closed in $H^1(\T)$ and hence complete; 
(ii) by Proposition~\ref{prop:poincare-vertical}, $\mathcal M$ is bounded in $H^1$; 
(iii) the embedding $H^1(\T)\hookrightarrow C^0(\T)$ is compact, so every sequence in $\mathcal M$ has a uniformly convergent subsequence.
\end{lemma}

\begin{proof}
(i) The defining symmetries of $\mathcal H_{\mathrm{eight}}$ are linear and closed in $H^1(\T)$; the constraint $K=1$ is closed (continuity of $K$ in $H^1$). Hence $\mathcal M$ is closed and, since $H^1$ is complete, so is $\mathcal M$. 
(ii) If $x\in\mathcal M$, then $\int|\dot\gamma|^2=2K=2$. 
By Proposition~\ref{prop:poincare-vertical}, $\int|\gamma|^2\le \int|\dot\gamma|^2\le 2$, 
so $\|x\|_{H^1}^2\le 4$, i.e.\ $\mathcal M$ lies in a fixed ball of $H^1$.

(iii) The embedding $H^1(\T)\hookrightarrow C^0(\T)$ is compact; hence every sequence in $\mathcal M$ admits a uniformly convergent subsequence (and a weakly convergent subsequence in $H^1$).
\end{proof}

\begin{lemma}[Lower semicontinuity of $V^q$ on $\mathcal M$]
\label{lem:lsc-Vq}
Extending $U$ by $+\infty$ at collisions, the resulting function is l.s.c.
If $x_n\to x$ in $H^1(\T)$, then (after extraction) $x_n\to x$ uniformly and
\[
V(x)\le \liminf_{n\to\infty}V(x_n)\quad\Rightarrow\quad V(x)^q\le \liminf_{n\to\infty}V(x_n)^q.
\]
In particular, $f:=V^q$ is l.s.c.\ and $f\ge 0$, hence bounded from below.
\end{lemma}
\begin{proof}
The statement relies on two standard facts: (a) the extension by $+\infty$ of a repulsive singular potential is lower semicontinuous, and (b) the compact embedding $H^1(\T)\hookrightarrow C^0(\T)$ ensures uniform convergence along a subsequence.

Write the extended potential \(\widetilde U:(\R^{2})^N\to[0,+\infty]\) as
\(\widetilde U(z)=U(z)\) if \(z\notin\Delta\) (no collision) and \(\widetilde U(z)=+\infty\) if \(z\in\Delta\) (some collision).
Then \(\widetilde U\) is l.s.c.\ (indeed, \(\widetilde U(z)\to+\infty\) as one approaches \(\Delta\)).
Let \(x_n\to x\) in \(H^1(\T)\). By the compact embedding \(H^1(\T)\hookrightarrow C^0(\T)\),
after extraction we may assume \(x_n\to x\) uniformly. For each \(t\),
pointwise l.s.c.\ gives
\[
\widetilde U\big(x(t)\big)\ \le\ \liminf_{n\to\infty}\widetilde U\big(x_n(t)\big).
\]
Since \(\widetilde U\ge 0\), Fatou’s lemma yields
\[
V(x)=\int_0^{2\pi}\widetilde U(x(t))\,dt\ \le\ \int_0^{2\pi}\liminf_{n}\widetilde U(x_n(t))\,dt
\ \le\ \liminf_{n}\int_0^{2\pi}\widetilde U(x_n(t))\,dt=\liminf_{n}V(x_n).
\]
As \(r\mapsto r^q\) is increasing and continuous for \(q>0\),
\(
V(x)^q\le \liminf_{n}V(x_n)^q.
\)
Hence \(f=\mathcal F|_{\mathcal M}=V^q\) is l.s.c.\ on \(\mathcal M\) with values in \([0,+\infty]\).
\end{proof}

Thus the pair $(\mathcal M,f)$ satisfies all assumptions of Ekeland’s variational principle.

\section{Ekeland’s principle: statement and application to $\mathcal F$}
\label{sec:ekeland}
We will use the following form of Ekeland’s variational principle.

\begin{theorem}[Ekeland’s principle]\label{thm:ekeland}
Let $(X,d)$ be a \emph{complete} metric space and 
$f:X\to(-\infty,+\infty]$ a lower semicontinuous function,
not identically $+\infty$, and bounded from below. 
Then, for any $\varepsilon,\rho>0$ there exists 
$x_{\varepsilon,\rho}\in X$ such that:
\begin{enumerate}[label=(\roman*),leftmargin=1.6em]
  \item $f(x_{\varepsilon,\rho})\le \inf_X f+\varepsilon$;
  \item $f(y)\ge f(x_{\varepsilon,\rho})-\tfrac{\varepsilon}{\rho}\,d(y,x_{\varepsilon,\rho})$ 
        for every $y\in X$.
\end{enumerate}

If, in addition, $X$ is a Hilbert space and $f$ is Fréchet–differentiable, 
there exists a sequence $(x_\varepsilon)\subset X$ such that 
\[
\|\nabla f(x_\varepsilon)\|\le \varepsilon,
\]
that is, the principle yields a \emph{Palais–Smale type} sequence 
whose values of $f$ approach the infimum.

Likewise, if $f$ is restricted to a differentiable submanifold 
$\mathcal M\subset X$ (for instance, the level set $\{K=1\}$), 
the same conclusion holds for the tangential gradient $\nabla_{\mathcal M}f$.
\end{theorem}

(See Ekeland \cite{Ekeland1979} for the full statement and its differentiable formulation in Hilbert spaces.)

\begin{proposition}[Existence of a minimizer via Ekeland]
\label{prop:existencia-ekeland}
In the complete metric space $(\mathcal M,d_{H^1})$, there exists $x_\ast\in\mathcal M$ such that $f(x_\ast)=\inf_{\mathcal M} f$.
Moreover, the differentiable form of Ekeland’s principle applied to $f|_{\mathcal M}$
provides a Palais–Smale type sequence in $\mathcal M$ which (after extraction) converges in $H^1$ and uniformly to $x_\ast$.
\end{proposition}

\begin{proof}
\emph{Step 1: Checking Ekeland’s hypotheses.}
By Lemma~\ref{lem:completez-M}, $(\mathcal M,d_{H^1})$ is complete.
By Lemma~\ref{lem:lsc-Vq}, $f=V^q$ is l.s.c.\ on $\mathcal M$ and $f\ge 0$, hence it is bounded from below
and is not identically $+\infty$ (there are smooth curves in $\mathcal H_{\mathrm{eight}}$ with $V<\infty$).

\emph{Step 2: Applying Ekeland and constructing a quasi–minimizing family.}
For $\varepsilon_n=1/n$ and $\rho=1$, Theorem~\ref{thm:ekeland} yields $x_n\in\mathcal M$ such that
\[
f(x_n)\le \inf_{\mathcal M} f+\tfrac{1}{n},
\qquad
f(y)\ge f(x_n)-\tfrac{1}{n}\,d_{H^1}(y,x_n)\quad\forall\,y\in\mathcal M.
\]
In particular $f(x_n)<\infty$, hence, using the extension of $U$ by $+\infty$ at collisions,
$x_n\in\Omega$, where $V$ (and thus $f$) is $C^1$.

\emph{Step 3: Palais–Smale type condition on the submanifold $\mathcal M$.}
Since \( \mathcal M=\{K=1\}\cap\mathcal H_{\mathrm{eight}} \) is a \(C^1\) submanifold of \(H^1\)
(the qualification \(DK\neq 0\) on \( \mathcal M \) is recorded in Remark~\ref{rem:cualificacion}),
we may apply the \emph{differentiable} form of Ekeland to the restricted functional \(f|_{\mathcal M}\).
We obtain a sequence \(\{x_n\}\subset\mathcal M\) with
\[
\|\nabla_{\mathcal M} f(x_n)\|\ \le\ \tfrac{1}{n}.
\]
Equivalently, there exists a Lagrange multiplier \(\mu_n\in\R\) such that
\[
\nabla f(x_n)+\mu_n\,\nabla K(x_n)\ \longrightarrow\ 0
\quad\text{in }(H^1)^\ast,
\]
namely, a Palais–Smale type condition for the restriction to \( \mathcal M \).

Since $\mathcal M=\{K=1\}$ is a $C^1$ codimension--one submanifold of the Hilbert space
$X=H^1(\T;(\R^2)^N)$, its tangent space at $x\in\mathcal M$ is
\[
T_x\mathcal M
=\{\eta\in X:\ DK(x)[\eta]=0\}.
\]
By the Riesz representation in $X$, the differential $DK(x)$ can be written as
\[
DK(x)[\eta]=\langle \nabla K(x),\eta\rangle_{H^1},
\]
for a unique vector $\nabla K(x)\in X$, so that
\[
T_x\mathcal M
=\{\eta\in X:\ \langle \nabla K(x),\eta\rangle_{H^1}=0\}
=\big(\operatorname{span}\{\nabla K(x)\}\big)^\perp.
\]
Thus we have the orthogonal decomposition
\[
X = T_x\mathcal M \oplus \operatorname{span}\{\nabla K(x)\}.
\]

For each $x\in\mathcal M$ and each $v\in X$, there exists a unique scalar $\mu\in\R$ such that
\[
v = v_T + \mu\,\nabla K(x),
\qquad v_T\in T_x\mathcal M.
\]
Applied to $v=\nabla f(x)$, this gives the decomposition
\[
\nabla f(x)
=\nabla_{\mathcal M} f(x)+\mu(x)\,\nabla K(x),
\]
where $\nabla_{\mathcal M} f(x)\in T_x\mathcal M$ is the tangential gradient of $f$ along $\mathcal M$.

In particular, for the sequence $\{x_n\}\subset\mathcal M$ given by Ekeland’s principle we may write
\[
\nabla f(x_n)+\mu_n\,\nabla K(x_n)=\nabla_{\mathcal M} f(x_n).
\]
Since Ekeland yields $\|\nabla_{\mathcal M} f(x_n)\|\le 1/n$, we obtain
\[
\|\nabla f(x_n)+\mu_n\,\nabla K(x_n)\|
=\|\nabla_{\mathcal M} f(x_n)\|\longrightarrow 0,
\]
which is precisely a Palais--Smale type condition for the restriction of $f$ to $\mathcal M$.

\emph{Step 4: Compactness and passage to the limit.}
Since \(K(x_n)=1\), Proposition~\ref{prop:poincare-vertical} provides a uniform bound on \( \|x_n\|_{H^1} \).
Extracting a subsequence, there exists \(x_\ast\) with
\[
x_n \rightharpoonup x_\ast \text{ in } H^1(\T),
\qquad
x_n \to x_\ast \text{ uniformly on } \T.
\]
The l.s.c.\ from Lemma~\ref{lem:lsc-Vq} gives
\(
f(x_\ast)\le \liminf_{n\to\infty} f(x_n)=\inf_{\mathcal M} f,
\)
hence \(x_\ast\) minimizes \(f\) on \(\mathcal M\).
Moreover, by continuity of \(K\), \(K(x_\ast)=1\), and uniform convergence preserves the symmetries, so \(x_\ast\in\mathcal H_{\mathrm{eight}}\).
Since \(f(x_\ast)<\infty\), the minimum cannot be at collision (the $+\infty$ extension would forbid it),
thus \(x_\ast\in\Omega\).

We conclude that there exists a minimizer \(x_\ast\) of \(f\) on \(\mathcal M\), and that Ekeland’s sequence
is of Palais–Smale type and (after extraction) converges to \(x_\ast\) in \(H^1\) and uniformly.
This \(x_\ast\) is the natural candidate to be a critical point of \(\mathcal F\) in \(\mathcal H_{\mathrm{eight}}\);
in the next sections we analyze its regularity and verify that, after reparametrization,
it solves Newton’s equation with no collisions.
\end{proof}

\section{Collision avoidance via Marchal’s criterion}
\label{subsec:ausencia-colisiones}

We use Marchal’s classical criterion to exclude collisions for action minimizers within symmetric classes that are stable under local variations.

\begin{lemma}[Marchal: no interior collisions]\label{lem:Marchal}
Let $U$ be a homogeneous potential of degree $-\alpha$ with $0<\alpha<2$ (including the Newtonian case $\alpha=1$).
If $x$ minimizes the Lagrangian action $\mathcal A=K+V$ in a class of curves stable under local variations preserving the global symmetries, then $x$ has no collisions in the interior of the period.
\emph{See} \cite{Marchal2002}.
\end{lemma}

In our setting, the minimizer provided by Ekeland’s principle for $f=V^q$ over
\[
\mathcal M=\{x\in\mathcal H_{\mathrm{eight}}:\ K(x)=1\},\qquad q=\tfrac{2}{\alpha+2},
\]
yields, after optimizing the scale, a minimizer of the action $\mathcal A$ in its symmetric class $(A_x)$–$(Y^\star)$,
which is stable under local variations (on subintervals respecting the symmetries). 
Therefore Lemma~\ref{lem:Marchal} applies directly: the minimizer is collision–free on $(0,2\pi)$.

\begin{remark}[Compatibility with the scale–invariant functional]\label{rem:marchal-f}
Recall that
\[
\mathcal F(x)=\min_{\lambda>0}\Phi_\alpha(\lambda;x),
\qquad
\Phi_\alpha(\lambda;x)=\lambda^2K(x)+\lambda^{-\alpha}V(x).
\]
For each $x$ let $\lambda^*(x)>0$ be the (unique) minimizer of $\Phi_\alpha(\cdot;x)$, and set
\[
y := x_{\lambda^*(x)}.
\]
Then, by the scaling laws for $K$ and $V$,
\[
\mathcal F(x)
=\min_{\lambda>0}\Phi_\alpha(\lambda;x)
=\Phi_\alpha\big(\lambda^*(x);x\big)
=\mathcal A(y).
\]

Suppose now that $y$ has a collision. 
By Marchal’s construction there exists a local variation $y^\sharp$ of $y$
(preserving the endpoints and global symmetries) such that
\[
\mathcal A(y^\sharp)<\mathcal A(y)=\Phi_\alpha\big(\lambda^*(x);x\big)=\mathcal F(x).
\]
Define $x^\sharp$ by $y^\sharp=(x^\sharp)_{\lambda^*(x)}$. 
We claim that $\mathcal F(x^\sharp)<\mathcal F(x)$, and we justify each step in the chain
\[
\mathcal F(x^\sharp)
=\min_{\lambda>0}\Phi_\alpha(\lambda;x^\sharp)
\le \Phi_\alpha\big(\lambda^*(x);x^\sharp\big)
=\mathcal A(y^\sharp)
< \mathcal A(y)
=\Phi_\alpha\big(\lambda^*(x);x\big)
=\mathcal F(x).
\]

\emph{(i) Definition of $\mathcal F(x^\sharp)$.}  By definition,
\[
\mathcal F(x^\sharp)=\min_{\lambda>0}\Phi_\alpha(\lambda;x^\sharp).
\]

\emph{(ii) Evaluation at a fixed scale.}  
For any fixed $\lambda>0$ one has
\[
\min_{\lambda>0}\Phi_\alpha(\lambda;x^\sharp)
\le \Phi_\alpha(\lambda;x^\sharp).
\]
In particular, taking $\lambda=\lambda^*(x)$ gives
\[
\min_{\lambda>0}\Phi_\alpha(\lambda;x^\sharp)
\le \Phi_\alpha\big(\lambda^*(x);x^\sharp\big).
\]
Here equality does not hold in general, since $\lambda^*(x)$ minimizes
$\Phi_\alpha(\cdot;x)$ but need not minimize $\Phi_\alpha(\cdot;x^\sharp)$.

\emph{(iii) Identification with the action of $y^\sharp$.}  
By definition of $y^\sharp=(x^\sharp)_{\lambda^*(x)}$ and of $\Phi_\alpha$,
\[
\Phi_\alpha\big(\lambda^*(x);x^\sharp\big)
= \mathcal A(y^\sharp).
\]

\emph{(iv) Marchal’s strict inequality.}  
Marchal’s deformation gives
\[
\mathcal A(y^\sharp)<\mathcal A(y).
\]

\emph{(v) Identification with $\mathcal F(x)$.}  
Since $\lambda^*(x)$ minimizes $\Phi_\alpha(\cdot;x)$, we have
\[
\mathcal A(y)
= \Phi_\alpha\big(\lambda^*(x);x\big)
= \mathcal F(x).
\]

Combining (i)–(v) yields $\mathcal F(x^\sharp)<\mathcal F(x)$.
Hence any curve with an interior collision can be locally modified to produce
a strictly smaller value of $\mathcal F$, so collision exclusion also holds
for minimizers of the scale–invariant functional $\mathcal F$.
\end{remark}

Because $\Phi_\alpha(\lambda;x)$ is convex in $\lambda$, any local decrease in $\mathcal A$ implies a strict decrease in $\mathcal F$ at the corresponding scale.

{\bf Binary collision versus regular path.}
For illustration, consider two equal masses undergoing a symmetric binary collision near $t=0$, with mutual distance 
$r(t)\sim c\,t^{2/(2+\alpha)}$ for some constant $c>0$. 
Then $|\dot r(t)|^2$ and $r(t)^{-\alpha}$ have the same local behaviour 
$t^{-2\alpha/(2+\alpha)}$, so that their contributions to the action over $(0,\varepsilon)$
scale like
\[
K_\varepsilon \sim \int_0^\varepsilon t^{-2\alpha/(2+\alpha)}\,dt,
\qquad
V_\varepsilon \sim \int_0^\varepsilon t^{-2\alpha/(2+\alpha)}\,dt
\sim C\,\varepsilon^{(2-\alpha)/(2+\alpha)},
\]
which remains finite for $0<\alpha<2$. 
Thus collision arcs have finite action, and a direct divergence argument cannot be used to discard them.
Instead, Marchal's construction shows that one can perform a local deformation near the collision ---preserving the endpoints and the global symmetries--- which strictly lowers the action $\mathcal A$. 
By Remark~\ref{rem:marchal-f}, this implies a strict decrease of the scale–invariant functional $\mathcal F$ as well, so collision paths cannot be minimizers of $\mathcal F$.


\begin{remark}[$G$–equivariant alternative]
Collisionlessness also follows from the general framework of Ferrario–Terracini \cite{FerrarioTerracini2004} for $G$–equivariant minimizers with the \emph{rotating circle property}. 
The class $(A_x)$–$(Y^\star)$ fits within that scheme (details omitted).
\end{remark}

Lemma~\ref{lem:Marchal} rules out interior collisions. 
In the periodic case, the endpoints $0$ and $2\pi$ are identified, and the reinforced Poincaré coercivity together with the symmetries prevents approach to total collision in the minimizing class while keeping $K=1$. 
Consequently, the variational principle suppresses configurations with unbounded action, energetically selecting smoother choreographies.

\section{Regularity of the critical point}
\label{sec:fig8}

\begin{lemma}[Regularity bootstrap: $x_\ast\in C^\infty$]
\label{lem:bootstrap}
The minimizer $x_\ast$ is collision--free.  
Therefore $U$ and all its derivatives are smooth and bounded along the trajectory.  
The weak Newton equation
\[
m\,\ddot x_\ast = \nabla U(x_\ast)
\]
then implies, via a standard iterative bootstrap argument, that 
$x_\ast\in C^\infty([0,2\pi])$; moreover, if $U$ is analytic away from collisions, 
then $x_\ast$ is real--analytic in time.
\end{lemma}

\begin{proof}
After the time–rescaling $s=\sqrt{\rho}\,t$ performed in the previous section,
the minimizer satisfies Newton’s equation in the classical form
\[
m\,\ddot x_\ast = \nabla U(x_\ast),
\]
where the dot denotes differentiation with respect to the rescaled time $s$.
We keep the notation $t$ for simplicity.

\paragraph{\textbf{Step 1: Weak formulation.}}
Since $x_\ast\in H^1(\T;\R^{2N})$ minimizes the action under the symmetry constraints,
for every test function $\eta\in C^\infty_{\mathrm{per}}([0,2\pi];\R^{2N})$ we have
\begin{equation}\label{eq:weak}
m\int_0^{2\pi} \dot x_\ast(t)\cdot \dot \eta(t)\,dt
=
\int_0^{2\pi} \nabla U(x_\ast(t))\cdot \eta(t)\,dt.
\end{equation}
Integrating by parts on the left (boundary terms vanish by periodicity), we obtain
\[
-\,m\int_0^{2\pi}\ddot x_\ast(t)\cdot\eta(t)\,dt
=
\int_0^{2\pi}\nabla U(x_\ast(t))\cdot\eta(t)\,dt
\qquad\forall\,\eta.
\]
Thus, in the sense of distributions,
\[
m\,\ddot x_\ast=\nabla U(x_\ast).
\]

\paragraph{\textbf{Step 2: $L^\infty$ regularity and $C^1$ regularity.}}
Because $x_\ast$ has no collisions, its image
\[
K := x_\ast([0,2\pi])
\subset (\R^2)^N\setminus\Delta
\]
is compact and lies at positive distance from the singular set $\Delta$.
Hence
\[
\nabla U|_K \in C^\infty(K),\qquad 
\|\nabla U\|_{L^\infty(K)}<\infty.
\]
Therefore the right-hand side $\nabla U(x_\ast(t))$ belongs to $L^\infty(0,2\pi)$, so
\[
\ddot x_\ast \in L^\infty(0,2\pi),\qquad
x_\ast\in W^{2,\infty}(0,2\pi).
\]

Since in one dimension 
\[
W^{2,\infty}(0,2\pi)\hookrightarrow C^{1,1}([0,2\pi]),
\]
we obtain
\[
x_\ast\in C^1([0,2\pi])\quad\text{and}\quad \dot x_\ast \text{ is Lipschitz}.
\]

\paragraph{\textbf{Step 3: Start of the bootstrap.}}
Having $x_\ast\in C^1$ and $\nabla U\in C^\infty$ on $K$, the chain rule implies
\[
t\longmapsto \nabla U(x_\ast(t)) \in C^1([0,2\pi]).
\]
Then the ODE
\[
m\,\ddot x_\ast = \nabla U(x_\ast)
\]
shows that $\ddot x_\ast\in C^1$, hence
\[
x_\ast\in C^3([0,2\pi]).
\]

\paragraph{\textbf{Step 4: Inductive step.}}
Assume $x_\ast\in C^r$ for some $r\ge1$.
Because $U\in C^\infty$ on $K$, all derivatives $\partial^\beta\nabla U$ are bounded on $K$.
Composing with $x_\ast\in C^r$ yields
\[
\nabla U(x_\ast)\in C^r.
\]
Thus $\ddot x_\ast\in C^r$, and integrating twice gives
\[
x_\ast\in C^{r+2}.
\]

\paragraph{\textbf{Step 5: Conclusion of the bootstrap.}}
Starting from $x_\ast\in C^1$ we obtain successively
\[
C^1 \longrightarrow C^3 \longrightarrow C^5 \longrightarrow \cdots
\]
and therefore 
\[
x_\ast\in C^\infty([0,2\pi];\R^{2N}).
\]

\paragraph{\textbf{Step 6: Analyticity.}}
If $U$ is real–analytic on $(\R^2)^N\setminus\Delta$,  
then $\nabla U$ is analytic, and the ODE 
\[
m\,\ddot x_\ast = \nabla U(x_\ast)
\]
is analytic on the compact collision–free set $K$.  
By the classical Cauchy–Kowalevski theorem for analytic ODEs, any $C^2$ solution whose image stays in $K$ is analytic in $t$.  
Since we have already shown $x_\ast\in C^\infty$ and $x_\ast([0,2\pi])\subset K$, it follows that $x_\ast$ is analytic.

\end{proof}

\subsection{Symmetries $(A_x)$–$(Y^\star)$: $C^1$ gluing, crossings, and uniqueness}

Consider the symmetries
\[
(A_x)\quad \gamma(t+\pi)=(x(t),-y(t)),
\qquad
(Y^\star)\quad \gamma\!\left(\tfrac{\pi}{2}+t\right)
   =\big(-x(\tfrac{\pi}{2}-t),\,y(\tfrac{\pi}{2}-t)\big),
\]
together with the phase normalization $\gamma(0)=(0,0)$.
These conditions define the symmetric class $\mathcal H_{\mathrm{eight}}$, 
characteristic of the vertical figure–eight.

\begin{lemma}[Boundary conditions and $C^1$ gluing in $\mathcal H_{\mathrm{eight}}$]
\label{lem:pegado-eight}
If $\gamma\in\mathcal H_{\mathrm{eight}}$, then:
\begin{enumerate}[leftmargin=1.6em]
\item $x(\tfrac{\pi}{2})=0$ and $y'(\tfrac{\pi}{2})=0$ 
      (the curve meets the $y$–axis orthogonally at $t=\tfrac{\pi}{2}$);
\item starting from the fundamental segment $[0,\tfrac{\pi}{2}]$, symmetry (Y$^\star$) 
      generates $[\tfrac{\pi}{2},\pi]$ with $C^1$ matching at $t=\tfrac{\pi}{2}$, 
      and (A$_x$) replicates this block on $[\pi,2\pi]$, completing the period.
\end{enumerate}
\end{lemma}

\begin{proof}
From (Y$^\star$) with $t=0$ we get
$\gamma(\tfrac{\pi}{2})=(-x(\tfrac{\pi}{2}),\,y(\tfrac{\pi}{2}))$, 
hence $x(\tfrac{\pi}{2})=-x(\tfrac{\pi}{2})$ and thus $x(\tfrac{\pi}{2})=0$.
Differentiating (Y$^\star$) and evaluating at $t=0$ gives
$\dot\gamma(\tfrac{\pi}{2})=(x'(\tfrac{\pi}{2}),-y'(\tfrac{\pi}{2}))$, 
so $y'(\tfrac{\pi}{2})=-y'(\tfrac{\pi}{2})$ and therefore $y'(\tfrac{\pi}{2})=0$. 
This ensures the $C^1$ gluing at $t=\tfrac{\pi}{2}$.
Symmetry (A$_x$) then pastes $[\pi,2\pi]$ with no loss of regularity, 
thanks to the parity/sign change of $x$ and $y$ under a shift by $\pi$.
\end{proof}

\begin{lemma}[Node at the origin]
\label{lem:nodo-origen-eight}
If $\gamma\in\mathcal H_{\mathrm{eight}}$ and we fix $\gamma(0)=(0,0)$, then
$\gamma(\pi)=(0,0)$ by (A$_x$). In particular, the origin is a self–intersection point
of the loop. Moreover, the velocity does not vanish at $t=0$ nor at $t=\pi$
(transverse crossings) for any nonconstant Newtonian solution obtained
by the variational procedure.
\end{lemma}

\begin{proof}
By (A$_x$), $\gamma(\pi)=(x(0),-y(0))=(0,0)$.
If $\dot\gamma(0)=0$ with $\gamma(0)=(0,0)$, local uniqueness for the Cauchy problem
would force a stationary solution, contradicting the nontriviality
of the minimizing orbit. The same argument applies at $t=\pi$ by symmetry.
\end{proof}

Lemma~\ref{lem:nodo-origen-eight} captures an intrinsic geometric feature
of any curve in the class $\mathcal H_{\mathrm{eight}}$:
the double transverse crossing at the origin, imposed solely by the symmetries.
This will be essential in Theorem~\ref{thm:fig8-eight},
since the variational minimizer preserves these symmetries
and therefore has the shape of a vertical figure–eight.

\begin{theorem}[Vertical figure–eight in $\mathcal H_{\mathrm{eight}}$ for $0<\alpha<2$]
\label{thm:fig8-eight}
Let $U$ be a homogeneous potential of degree $-\alpha$ with $0<\alpha<2$ and three equal masses in the plane.
Then there exists a $2\pi$–periodic three–body solution of choreographic type \eqref{eq:choreo},
belonging to the class $\mathcal H_{\mathrm{eight}}$, collision–free and with a node at the origin
at times $t\equiv 0\ (\mathrm{mod}\ \pi)$.
After a suitable time reparametrization, it solves the Newton equation associated with $U$.
\end{theorem}

\begin{proof}
\emph{Step 1 (scale invariance and compactness).}
By the homothety invariance of $\mathcal F$ and the equivalence of minima
between $\mathcal F$ and $V^q$ with $q=2/(\alpha+2)$, 
it suffices to minimize $V^q$ over
\[
\mathcal M:=\{x\in\mathcal H_{\mathrm{eight}}:\ K(x)=1\}.
\]
The reinforced Poincaré inequality (Proposition~\ref{prop:poincare-vertical})
controls $\|x\|_{L^2}$ in terms of $K$, 
and together with completeness of $\mathcal M$ in $H^1$ and the compact embedding 
$H^1(\T)\hookrightarrow C^0(\T)$, gives the required compactness.

\emph{Step 2 (Ekeland minimizer).}
Applying Ekeland’s principle on $(\mathcal M,d_{H^1})$ yields a 
minimizer $x_\ast\in\mathcal M$ of $f=V^q$.

\emph{Step 3 (no collisions and regularity).}
By Marchal’s argument (or the Ferrario–Terracini framework) 
adapted to classes invariant under local reflections,
$x_\ast$ has no collisions in the interior of the period. 
Lemma~\ref{lem:bootstrap} then gives $x_\ast\in C^\infty$.

\emph{Step 4 (Euler–Lagrange and Newton).}
There exist constants $m>0$ and $\rho>0$ such that
\[
m\,\ddot x_\ast=\rho\,\nabla U(x_\ast),\qquad 2K=\alpha\,\rho\,V,
\]
with $\rho$ constant. Reparametrizing time by $s=\sqrt{\rho}\,t$
we obtain the standard Newton equation
\( m\,\tfrac{d^2x_\ast}{ds^2}=\nabla U(x_\ast(s)) \).

\emph{Step 5 (geometry and node).}
The symmetries (A$_x$) and (Y$^\star$), together with 
Lemma~\ref{lem:pegado-eight}, reconstruct the whole orbit 
from the fundamental segment $[0,\pi/2]$, 
which meets the $y$–axis orthogonally at $t=\pi/2$. 
Lemma~\ref{lem:nodo-origen-eight} ensures the transverse node at $t=0$ and $t=\pi$,
thus producing a vertical figure–eight.

\emph{Step 6 (choreography).}
Setting $x_i(t)=\gamma(t+\tfrac{2\pi i}{3})$, $i=0,1,2$, 
gives the choreography \eqref{eq:choreo} 
which satisfies Newton’s equations after the time reparametrization.
\end{proof}

The family of minimizers depends smoothly on $\alpha\in(0,2)$, providing a continuous deformation between weak and Newtonian interaction regimes.

\begin{corollary}[Newtonian case $\alpha=1$]
\label{cor:newton-eight}
Under the hypotheses of Theorem~\ref{thm:fig8-eight} with $\alpha=1$,
one obtains a Newtonian $2\pi$–periodic, choreographic, collision–free solution
in the class $\mathcal H_{\mathrm{eight}}$, with a node at the origin for
$t\equiv 0\ (\mathrm{mod}\ \pi)$.
This is precisely the configuration discovered by 
Chenciner and Montgomery (2000).
\end{corollary}

As shown in Figure~\ref{fig:fig8colinear}, the three equal masses are aligned along the $x$–axis, corresponding to a collinear configuration of the figure–eight choreography for $\alpha = 3/2$.

\begin{figure}[ht]
  \centering
  \includegraphics[width=0.5\textwidth]{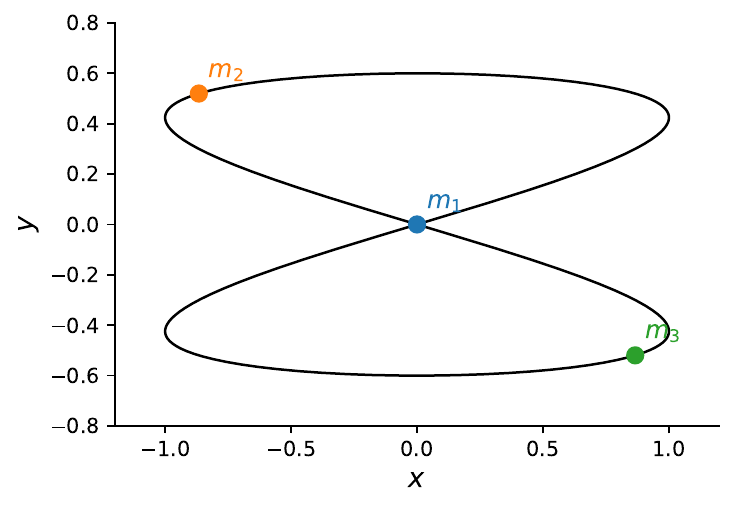}
  \caption{Idealized figure–eight choreography for three equal masses in the symmetric class $\mathcal H_{\mathrm{eight}}$, satisfying $(A_x)$ and $(Y^\star)$, $\alpha=3/2$. 
  The configuration shown corresponds to a collinear instant of the periodic motion.}
  \label{fig:fig8colinear}
\end{figure}

\section*{Conclusion}
The functional $\mathcal F=K^{\alpha/(\alpha+2)}V^{2/(\alpha+2)}$ provides a unified variational framework for all homogeneous potentials with $0<\alpha<2$. 
It allows a direct application of Ekeland’s principle on a compact symmetric manifold, recovering the classical Newtonian figure–eight and extending it to a continuous family of choreographies. 
The same variational setting may be adapted to investigate other symmetry types and larger values of $N$, where the scale–invariant formulation could simplify existence proofs and clarify the relation between symmetry constraints and collision avoidance. 
This perspective opens a path toward a systematic classification of variational choreographies beyond the three–body case.

\section*{Acknowledgments}
The authors wish to thank Prof.~Martha Álvarez-Ramírez for her valuable comments and observations, which greatly contributed to improving the clarity and organization of the manuscript.


\end{document}